\documentclass[12pt,reqno]{amsart}

\usepackage{amsfonts,amsmath,amsthm}
\usepackage{cases}
\usepackage{amssymb,epsfig}
\usepackage{enumerate} 

\usepackage{color} 
\definecolor{vert}{rgb}{0,0.6,0}

\usepackage[text={425pt,650pt},centering]{geometry}
\usepackage{graphicx}
\usepackage{epsfig}
\usepackage{tikz}
\usepackage{caption}
\usepackage{color} 
\definecolor{vert}{rgb}{0,0.6,0}

\numberwithin{figure}{section}

\theoremstyle{plain}
\newtheorem{thm}{Theorem}[section]

\newtheorem{defn}{Definition}

\newtheorem{ex}{Example}
\newtheorem{lem}[thm]{Lemma}
\newtheorem{cor}[thm]{Corollary}
\newtheorem{prop}[thm]{Proposition}
\theoremstyle{remark}
\newtheorem{rem}{\bf{Remark}}
\numberwithin{equation}{section}

\newcommand{\R}{\mathbb{R}}

\newcommand{\cA}{\mathcal{A}}

\newcommand{\cC}{\mathcal{C}}

\newcommand{\AC}{{\rm AC\,}}

\newcommand{\BUC}{{\rm BUC\,}}
\newcommand{\UC}{{\rm UC\,}}

\newcommand{\Lip}{{\rm Lip\,}}

\newcommand{\gam}{\gamma}
\newcommand{\del}{\delta}
\newcommand{\ep}{\varepsilon}
\newcommand{\kap}{\kappa}
\newcommand{\lam}{\lambda}

\newcommand{\Del}{\Delta}
\newcommand{\Gam}{\Gamma}

\newcommand{\ol}{\overline}
\newcommand{\ul}{\underline}

\newcommand{\supp}{{\rm supp}\,}

\newcommand{\Div}{{\rm div}\,}

\begin{document}
\title[On large time behavior]{Remarks on large time behavior of\\ 
level-set mean curvature flow equations
\\ with driving and source terms}
\author[Y. GIGA, H. MITAKE, H. V. TRAN]
{Yoshikazu Giga, Hiroyoshi Mitake, Hung V. Tran}

\thanks{
The work of YG was partially supported by Japan Society for the Promotion of Science (JSPS) through grants KAKENHI \#19H00639, \#18H05323, \#26220702,  \#16H03948.
The work of HM was partially supported by the JSPS grants: KAKENHI \#19K03580, \#19H00639, \#16H03948.
The work of HT is partially supported by NSF grant DMS-1664424 and NSF CAREER grant DMS-1843320.
}

\address[Y. Giga, H. Mitake]{
Graduate School of Mathematical Sciences, 
University of Tokyo 
3-8-1 Komaba, Meguro-ku, Tokyo, 153-8914, Japan}
\email{labgiga@ms.u-tokyo.ac.jp, mitake@ms.u-tokyo.ac.jp}


\address[H. V. Tran]
{
Department of Mathematics, 
University of Wisconsin Madison, Van Vleck Hall, 480 Lincoln Drive, Madison, Wisconsin 53706, USA}
\email{hung@math.wisc.edu}

\date{\today}



\maketitle

\begin{abstract}
We study a level-set mean curvature flow equation  with driving and source terms, and  
establish convergence results on the asymptotic behavior of solutions as time goes to infinity under some additional assumptions. 
We also study the associated stationary problem in details in a particular case, 
and establish Alexandrov's theorem in two dimensions in the viscosity sense, which is of independent interest.     
\end{abstract}

\section{Introduction}
In this paper, we study the large time behavior of the viscosity solution $u$ to a 
degenerate parabolic PDE of the form 
\[
{{\rm(C)}\qquad}
\begin{cases}
\displaystyle
u_t-\left(\Div\Big(\frac{Du}{|Du|}\Big)+1\right)|Du|=f(x) \quad &\text{ in } \R^n \times(0,\infty),\\
u(\cdot,0)=u_0 \quad &\text{ on } \R^n, 
\end{cases}
\]
where $n\geq 2$, and $f:\R^n \to \R$ is the source term, which is nontrivial and  satisfies
\begin{equation}\label{f-source}
f \in \Lip(\R^n), \ f\ge0, \text{ and } \supp(f)\subset B(0,R)  \text{ for some} \ R>0.
\end{equation}
In other words, $f \geq 0$ is Lipschitz continuous on $\R^n$ such that the 
support of $f$ is contained in an open ball $B(0,R)$ of radius $R>0$ centered 
at the origin.
The initial condition $u_0:\R^n \to \R$ is in $\BUC(\R^n)$, the space of all bounded uniformly continuous functions on $\R^n$.
Here, $u:\R^n\times[0,\infty)\to\R$ is a unknown function, and 
$u_t$, $Du$ denote the time derivative and the spatial gradient of $u$, respectively. 
Note that if the source term $f$ is identically equal to zero, the equation (C) is nothing 
but the level set flow equation of the motion by 
mean curvature plus a constant ($=1$) driving force. Its analytic 
foundation like well-posedness and comparison principle has been well 
established by adjusting the theory of viscosity solutions \cite{CGG, ES1}. 
See also e.g. \cite{G-book} and references therein for further development.
We are always concerned with viscosity solutions in this paper, and the term ``viscosity" is omitted henceforth.

Equation (C) was derived as a continuum limit of the birth and spread model in \cite{GMT} (see also \cite{GMOT}) with a motivation to describe crystal growth in supersaturated environments.   
In \cite{GMOT}, we obtained the existence of the large time average of $u$ which was denoted by the \textit{asymptotic speed}.
More precisely, we showed that, locally uniformly for $x\in \R^n$,
\begin{equation}\label{asymp-speed}
c=\lim_{t\to\infty} \frac{u(x,t)}{t} 
\end{equation}
exists under assumption \eqref{f-source}. 
We then studied qualitative properties of asymptotic speed $c=c_f$ in some specific settings. 

In this paper, as a next step, we investigate the large time asymptotic of $u$, that is, behavior of $u(x,t)$ as $t \to \infty$.
We establish the convergence result
\begin{equation}\label{conv}
u(x,t)-ct\to v_\infty(x) \quad\text{locally uniformly for } x \in \R^n \ \text{as} \ t\to\infty,  
\end{equation}
where $v_\infty$ is a solution to 
\begin{equation}\label{eq:static}
-\left(\Div\Big(\frac{Dv}{|Dv|}\Big)+1\right)|Dv|=f(x)-c\quad\text{in} \ \R^n  
\end{equation}
in two specific cases.  
If \eqref{conv} holds, we  denote by $v_\infty$ the \textit{asymptotic profile} of $u$ with given initial data $u_0$. 
The question on whether \eqref{conv} holds or not was addressed in \cite{Giga-ICM}, and it seems quite difficult to be resolved in the most general setting at this moment.

For asymptotic profiles, it is easy to derive scaled one of the form 
\[
\lim_{\lam\to\infty} \frac{u(\lam x,\lam t)}{\lam}= \max\{c(t-|x|), 0\}
\]
as stated in \cite{Giga-ICM}. See also \cite{H} for the first order problem. 

\subsection{Main results}
We now describe two specific cases in which we are able to obtain \eqref{conv}.
In Case 1, we consider the radially symmetric setting, that is, we assume the following.
\begin{itemize}
\item[(A1)]   $f(x) = \tilde f(|x|)$ for all $x\in \R^n$, where  $\tilde f\in C_c^1([0,\infty))$, 
i.e., $f\in C^1([0,\infty))$ has a compact support in $[0,\infty)$.

\item[(A2)] $u_0(x)=\tilde{u}_0(|x|)$ for all $x\in \R^n$, where $\tilde u_0 \in \BUC([0,\infty))$.  
\end{itemize}

\medskip

\noindent In Case 2, we put a different set of assumptions.
\begin{itemize}
\item[(A3)] $f\in C^1_c(\R^n)$ and the set $U:=\{x\in \R^n\,:\, f(x)=\max_{\R^n}f\}$ contains $B(x_0,n-1)$ for some $x_0\in\R^n$.

\item[(A4)] $U=\bigcup_{i\in I} E_i$, where $I \neq \emptyset$ is an index set, and for each $i\in I$, $E_i \subset \R^n$ is a closed set with nonempty interior and $C^2$-boundary satisfying $\kappa+1 \geq 0$ on $\partial E_i$.
Here, for $x\in \partial E_i$, $\kappa(x)$ denotes the mean curvature of $\partial E_i$ in the direction of the outer normal vector to $E_i$ at $x$.

\item[(A5)] $u_0 \equiv 0$.
\end{itemize}

\smallskip

\noindent Here are our main convergence results.

\begin{thm}\label{thm:main1}
Assume  {\rm (A1)--(A2)}. 
Let $u$ be the solution to {\rm (C)}.
Then \eqref{conv} holds, where $v_\infty$ is a solution to \eqref{eq:static}.
\end{thm}

\begin{thm}\label{thm:nonrad}
Assume  that {\rm  (A3)--(A5)} hold. 
Let $u$ be the solution to {\rm (C)}.
We have \eqref{conv}, where $v_\infty$ is a solution to \eqref{eq:static}. 
\end{thm}
 
In Case 1, we have a further characterization of $v_\infty$ which is stated as  following.
\begin{thm}\label{thm:main2}
Assume that {\rm(A1)--(A2)} hold. 
Let $u$ be the solution to {\rm (C)}.
Then, the asymptotic speed $c$, and asymptotic profile $v_\infty$ are 
described by $c=\max_{|x|\ge1}f(x)$, and $v_\infty(x)=\psi_\infty(|x|)$ for $x\in \R^n$, 
where $\psi_\infty$ is given by 
\begin{equation}\label{func:psi}
\psi_\infty(r)=\max\left\{d(r,s)+v_0(s)\,:\, s\in\tilde{\cA} \right\}.   
\end{equation}
Here, 
\begin{equation}\label{func:d}
d(r,s):=
\sup_{\gam\in\cC(t,0;r,s), t>0} 
\left\{ \int_0^t\left( \tilde{f}(\gam(z))-c\right)\,dz\,:\,
 \gam([0,t]) \subset (0,\infty)\right\}  
\end{equation}
for any $r,s\in[0,\infty)$, 
where we set 
\begin{multline*}
\cC(t,0;r,s):=\{\gam\in\AC([0,t]; (0,\infty))\,:\,
 \gam(t)=r,\ \gam(0)=s, \\ 
 \left|\gam'(z)+\frac{n-1}{\gam(z)}\right| \leq 1 \ \text{ for a.e. } z\in (0,t) 
\},  
\end{multline*}
and 
\begin{align*}
&v_0(r):=\max\left\{d(r,\rho)+\tilde{u}_0(\rho)  \,:\,  \rho\in[0,\infty)\right\}, \ 
\tilde{\cA}:= \left\{r \geq n-1 \,:\,  \tilde{f}(r)=c \right\}. 
\end{align*}
\end{thm}

It is worth noting that the set $\tilde \cA$ is the set of all equilibrium points associated with 
our cost function. 
Moreover,  $\tilde \cA$ plays a role of a \textit{uniqueness set} for 
the associated stationary problem, which is stated in Theorem \ref{thm:unique}.

\subsection{Literature and discussions}
In the last two decades, 
a lot of works have been devoted to the study of large time behavior of 
viscosity solutions of both first-order and second-order Hamilton-Jacobi equations.
In these papers, the Hamiltonian $H=H(x,p)$ is assumed to be convex and coercive in $p$.
See \cite{NR, Fa, BS, DS, Is, II, CGMT, BLNP, LMT, GH} and the references therein for a complete description of the literature. 
Nevertheless, to the best of our knowledge, large time behavior of (C) has not been studied.
It is worth noting that in general, (C) is not convex/concave in the gradient variable, and the methods in aforementioned results do not apply here.
Furthermore, the spatial variable $x$ is in the whole $\R^n$ without any periodicity, which is not compact, and hence, another layer of difficulty appears.
See \cite{GOS, CN, HM} for some results on  large time behavior on the mean curvature flow in different contexts.

In the two cases described above, we are able to use specific structure of the PDE to obtain large time behavior.
More specifically, in the radially symmetric setting (Case 1), it is intuitively clear that solution $u$ to (C) is also radially symmetric, that is,
$u(x,t)= \phi(|x|,t)$ for all $(x,t) \in \R^n \times [0,\infty)$ for some $\phi:[0,\infty)\times [0,\infty) \to \R$. 
It turns out that $\phi$ solves
\begin{equation}\label{sing-HJ}
{\phi}_t - \frac{n-1}{r} {\phi}_r - |{\phi}_r| = f(r) \quad \text{ in } (0,\infty) \times (0,\infty),
\end{equation}
a first-order Hamilton-Jacobi equation that is singular at $r=0$.
Set 
\[
H(r,p):= -\frac{n-1}{r}p - |p| - f(r) \quad \text{ for } (r,p) \in (0,\infty) \times \R.
\]
It is clear that $H$ is concave but is not coercive in $p$.
We therefore need to treat the problem with care, and indeed, the behavior of solutions at $r \leq n-1$ is rather different from the behavior at $r>n-1$.
Some of our ideas are inspired by \cite{NR, DS}. 

The treatment of Case 2 is quite different. For this, we aim at showing that $U$ is a uniqueness set of \eqref{eq:static} (to be described carefully later), and then use geometric properties of each set $E_i$ for $i\in I$ to conclude.

\subsection*{Organization of the paper} 
In Section \ref{sec:case1}, we study the radially symmetric setting,
and give proofs to Theorem \ref{thm:main1}, and Theorem \ref{thm:main2}.
Section \ref{sec:case2} is devoted to prove Theorem \ref{thm:nonrad}.
In particular, we show that $U$ is a uniqueness of ergodic problem \eqref{eq:static} in Proposition \ref{prop:unique}.
Finally, in Section \ref{sec:further}, we discuss further on the uniqueness set $U$ under assumption (A3), and show that the matter is quite complex in general 
as ergodic problem \eqref{eq:static} restricted to $U$ (see \eqref{eq:v-U}) might have infinitely many solutions.
Therefore, if we do not assume more ((A4)--(A5)), it is not yet known whether convergence \eqref{conv} holds or not in the nonradially symmetric setting.
We also establish Alexandrov's theorem in two dimensions in the viscosity sense, which is of independent interest.

\section{Case 1: radially symmetric case} \label{sec:case1}
In this section, we \textit{always} assume that (A1)--(A2) hold. 
In the radially symmetric case, 
the unique solution $u$ to (C) is represented by 
$u(x,t)=\phi(|x|,t)$ 
where $\phi$ is defined as 
\begin{multline}\label{rep-phi}
\phi(r,t)=\sup \Big\{ \int_0^t \tilde{f}(\gam(s))\,ds+\tilde{u}_0(\gam(0))\,:\, \gam([0,t]) \subset (0,\infty), \\
 \gam(t)=r,\ \left|\gam'(s)+\frac{n-1}{\gam(s)}\right| \leq 1 \ \text{for a.e. } s\in [0,t] \Big\} 
\end{multline}
for $(r,t) \in (0,\infty) \times [0,\infty)$. 
Note that $\phi\in\UC((0,\infty)\times[0,\infty))$. See \cite{GMT, GMOT}.

By using  formula \eqref{rep-phi}, it was obtained in \cite{GMOT} that,
locally uniformly for $x\in \R^n$,
\begin{equation}\label{asym-speed}
\frac{u(x,t)}{t}\to c:=\max_{|x|\ge n-1}f(x) = \max_{r\geq n-1} \tilde f(r)
\quad  \text{as} \ t\to\infty.  
\end{equation}
We now aim at studying finer behavior of $u(x,t)$ as $t \to \infty$.

%

\subsection{Ergodic problem}
In this subsection, we study 
\begin{equation}\label{eq:st-rad}
- \frac{n-1}{r} \psi_r - |\psi_r| = h(r)  \quad \text{ in } \ (0,\infty),
\end{equation}
where we set $h(r):=\tilde f(r)-c$ for convenience. 
We first establish the existence result of solutions to \eqref{eq:st-rad}. 
\begin{thm}\label{thm:ex}
There exists a solution $\psi\in C([0,\infty))$ to \eqref{eq:st-rad} which is bounded from above. 
Moreover, $\psi$ has a linear growth at infinity.
More precisely, 
\[
\lim_{r\to\infty}\frac{\psi(r)}{1+r} \ \text{exists and  is negative.} 
\] 
\end{thm}

\begin{proof}
We construct $\psi$ explicitly as following.
Let
\[
r_0 = \min \left\{r \geq n-1 \,:\,  \tilde{f}(r)=c \right\} = \min \left\{r \geq n-1 \,:\,  h(r)=0 \right\}.
\]
We first define $\psi(r)$ for $r\in [n-1,\infty)$.
For $r>r_0$, we let $\psi_r(r) \leq 0$ and hence,
\begin{equation}\label{def:neg}
\psi_r(r) = \frac{r h(r)}{r-(n-1)} \quad \text{ for } r>r_0.
\end{equation}

There are two cases to be considered here.
Firstly, if $r_0>n-1$, then we are able to extend the above definition for $r=r_0$ naturally, and clearly, $\psi_r(r_0)=0$.
For $r\in [n-1,r_0)$, we let $\psi_r(r) \geq 0$, that is,
\begin{equation}\label{def:pos}
\psi_r(r) = \frac{-r h(r)}{r+(n-1)} \quad \text{ for } n-1\leq r <r_0.
\end{equation}
Secondly, if $r_0=n-1$, we could also extend \eqref{def:neg} and define that
\[
\psi_r(n-1) = \lim_{r \to n-1} \frac{r h(r)}{r-(n-1)}= \lim_{r \to n-1} \frac{r (h(r)-h(n-1))}{r-(n-1)} = (n-1) h'(n-1) \leq 0.
\]
In both cases, it is clear that $\psi \in C^1([n-1,\infty))$ and $\psi$ solves \eqref{eq:st-rad} classically in $(n-1,\infty)$.

Next, we define $\psi(r)$ for $r\in [0,n-1)$.
For this, we decompose the interval $I= (0,n-1)$ into 
\begin{align}
I &=  (0,n-1) \nonumber \\
  &= \{r \in I\,:\, h(r)<0\} \cup \{r \in I\,:\, h(r)= 0\} \cup \{r \in I\,:\, h(r) > 0\} \nonumber \\
  &=:  A \cup B \cup C. \label{ABC}
\end{align}
For $r\in A$, we let $\psi_r(r) \geq 0$ and define it as in \eqref{def:pos}.
For $r \in B$, we simply let $\psi_r(r)=0$.
And finally, for $r\in C$, we set $\psi_r(r) \leq 0$ as in \eqref{def:neg}.
In all cases, for $r \in (0,1/2)$,
\[
|\psi_r(r)| \leq 2r\|h\|_{L^\infty},
\]
which allows us to extend naturally that $\psi_r(0)=0$. 
Here $\|h\|_\infty$ denotes the (essential) supremum norm of $h$. 
Finally, set $\psi(0)=0$.

It is clear from the construction that $\psi \in C^1([0,\infty))$ is bounded from above and 
satisfies the growth condition, and furthermore $\psi$ solves \eqref{eq:st-rad} classically in $(0,\infty)$.
\end{proof}
For $\psi$ constructed above, we define
\[
v(x)=\psi(|x|) \quad \text{ for all } x\in \R^n.
\]
Then, $v$ is a solution to \eqref{eq:static} (see \cite{GMOT} for a detailed proof).
The next lemma concerns behavior of solutions to \eqref{eq:st-rad} in $(0,n-1)$.

\begin{lem}\label{lem:unique1}
Let $A, C\subset(0,n-1)$ be the sets defined by \eqref{ABC}, 
and $\psi$ be an arbitrary solution to \eqref{eq:st-rad} and $(a,b) \subset (0,n-1)$. 
Then, the following holds.
\begin{itemize}
\item[(i)] If $(a,b) \subset A$, then $\psi_r(r) \geq 0$ and $\psi_r(r)$ satisfies \eqref{def:pos} for $r \in (a,b)$.
\item[(ii)] If $(a,b) \subset C$, then $\psi_r(r) \leq 0$ and $\psi_r(r)$ satisfies \eqref{def:neg} for $r\in (a,b)$.
\item[(iii)]
Let $\psi_1, \psi_2$ be two solutions to \eqref{eq:st-rad}.
Then, $(\psi_1)_r(r) = (\psi_2)_r(r)$ for a.e. $r\in [0,n-1]$.
In particular, $\psi_1-\psi_2$ is constant on $[0,n-1]$. 
\end{itemize}
\end{lem}

\begin{proof}
We proceed to prove (i).
Assume by contradiction that $\psi$ is not increasing in $(a,b)$, then, we can find $a<c<d<b$ such that $\psi(c)>\psi(d)$.

If $\psi$ has a local maximum at $y \in (c,d)$, then by the subsolution test at $y$, we get $0 \leq h(y) <0$, which is absurd.
Thus, $\psi$ does not have a local maximum in $(c,d)$, and hence, $\psi$ is decreasing in $(c,d)$.
We then have that $\psi_r(r) \leq 0$ for a.e. $r\in (c,d)$. We take one such $r$ and plug it into \eqref{eq:st-rad} to yield
\[
0>h(r)=-\frac{n-1}{r} \psi_r(r) - |\psi_r(r)| = \left(\frac{n-1}{r} - 1\right) |\psi_r(r)| \geq 0,
\]
which is also absurd.

Therefore, $\psi$ is increasing in $(a,b)$ and $\psi_r(r) \geq 0$ for a.e. $r\in (a,b)$.
It is then straightforward to see that $\psi_r(r)$ satisfies \eqref{def:pos} for all $r\in (a,b)$.

Since the proof of (ii) is analogous to the above, and we can get (iii) as a straightforward result of (i)--(ii), we omit them. 
\end{proof}

We are now interested in behavior of solutions to \eqref{eq:st-rad} in $(n-1,\infty)$.
\begin{lem}\label{lem:unique3}
Let $\psi_1, \psi_2$ be two solutions to \eqref{eq:st-rad}.
Pick $a,b \in \tilde \cA$ such that $(a,b) \cap \tilde \cA = \emptyset$.
Assume that $\psi_1(a)=\psi_2(a)$ and $\psi_1(b)=\psi_2(b)$.
Then, $\psi_1=\psi_2$ in $(a,b)$.
\end{lem}

\begin{proof}
We note that, for $r\in (a,b)$,
\[
\widetilde H(r,p):= -\frac{n-1}{r}p - |p| \leq \left(\frac{n-1}{r} - 1 \right) |p| = - \frac{r-(n-1)}{r} |p|  \leq 0,
\]
which tends to $-\infty$ and $|p| \to \infty$.
Hence, $\psi_1, \psi_2$ are locally Lipschitz in $(a,b)$ and differentiable a.e. there.
Of course, at places of differentiability of $\psi_i$ for $i=1,2$, $(\psi_i)_r(r)$ takes value of either form \eqref{def:pos} or \eqref{def:neg}.

Besides, as $p\mapsto \widetilde H(r,p)$ is concave, and $h(r)<0$ in $(a,b)$, 
we conclude that the graphs of $\psi_1, \psi_2$ cannot have corners from above, and can only have at most one corner from below in $(a,b)$.
Indeed, if $\psi_i$ has a corner from above at $z \in (a,b)$ for some $i \in \{1,2\}$, then it means that $\psi_i$ has a local maximum at $z$.
We can then apply the viscosity subsolution test at $z$ to deduce that $0 \leq h(z)<0$, which is absurd.
The fact that $\psi_i$ cannot have corners from above implies immediately that it has at most one corner from below.

Thus, for each $i=1,2$, there exists at most one point $c_i \in (a,b)$ such that
\[
\begin{cases}
(\psi_i)_r \leq 0 \text{ and } (\psi_i)_r \text{ satisfies \eqref{def:neg} in } (a,c_i),\\
(\psi_i)_r \geq 0 \text{ and } (\psi_i)_r \text{ satisfies \eqref{def:pos} in } (c_i,b).
\end{cases}
\]
Since $\psi_1(a)=\psi_2(a)$ and $\psi_1(b)=\psi_2(b)$, in case that $c_1, c_2$ exist, it is clear that $c_1=c_2$.
We thus get the desired conclusion.
\end{proof}

\begin{lem}\label{lem:unique4}
Let $\psi_1, \psi_2$ be two bounded from above solutions to \eqref{eq:st-rad}.
Assume that $\psi_1(r_0)=\psi_2(r_0)$ and $\psi_1(M)=\psi_2(M)$, where
$r_0=\min\{r\in[n-1,\infty)\,:\, h(r)=0\}$ and $M=\max \{r\in[n-1,\infty)\,:\, h(r)=0\}$.  
Then, $\psi_1=\psi_2$ on $[n-1,r_0]\cup [M,\infty)$. 
\end{lem}
\begin{proof}
We first show that $\psi_1=\psi_2$ on $[n-1,r_0]$. If $r_0=n-1$, then there is nothing to prove.
We therefore only need to consider the case that $r_0>n-1$.
As noted in the proof above, $\psi_i$ cannot have corners from above, and thus,
\[
(\psi_i)_r(r) = \frac{-r h(r)}{r+(n-1)} >0 \quad \text{ for } n-1 \leq r < r_0, \, i=1,2.
\]
We use this and the hypothesis that $\psi_1(r_0)=\psi_2(r_0)$ to yield $\psi_1=\psi_2$ on $[n-1,r_0]$.

\medskip

We next prove that $\psi_1=\psi_2$ on $[M,\infty)$. 
To get this, we aim at showing 
\begin{equation}\label{claim-neg}
(\psi_1)_r =(\psi_2)_r \leq 0 \quad \text{ on } [M,\infty).
\end{equation}
Indeed, as $\psi_1, \psi_2$ cannot have corners from above, and can only have at most one corner from below, if \eqref{claim-neg} were false, then there would exist $i \in \{1,2\}$, and $y \geq M$ such that 
\[
(\psi_i)_r(r)  = \frac{-r h(r)}{r+(n-1)} >0 \quad \text{ for all } r>y.
\]
This implies that $\psi_i$ is not bounded from above, which is absurd.
The proof is complete.
\end{proof}

Combining the results of Lemmas \ref{lem:unique1}--\ref{lem:unique4}, we 
obtain our main result in this subsection. 
\begin{thm}\label{thm:unique}
The following holds.
\begin{itemize}
\item[(i)] Let $\psi_1, \psi_2$ be solutions to \eqref{eq:st-rad}, which are bounded from above.
Assume further that $\psi_1=\psi_2$ on $\tilde \cA$. 
Then $\psi_1=\psi_2$ on $[0,\infty)$.

\item[(ii)] Let $\xi_1, \xi_2$ be, respectively, a subsolution, and a supersolution to \eqref{eq:st-rad}, which are bounded from above.
Assume further that $\xi_1 \leq \xi_2$ on $\tilde \cA$.
Then $\xi_1 \leq \xi_2$ on $[0,\infty)$.

\end{itemize}
\end{thm}

\subsection{Large time behavior}
We now study the behavior of $u(x,t) = \phi(|x|,t)$ as $t \to \infty$.

\begin{proof}[Proof of Theorem {\rm{\ref{thm:main1}}}]
Let $\tilde{\phi}(r,t) = \phi(r,t) - ct$ for all $(r,t) \in [0,\infty) \times [0,\infty)$, 
where $c$ is given by \eqref{asym-speed}. Then $\tilde{\phi}$ solves
\[
\tilde{\phi}_t - \frac{n-1}{r} \tilde{\phi}_r - |\tilde{\phi}_r| = h(r) \quad \text{ in } (0,\infty) \times (0,\infty).
\]
Noting that  
\[
-\frac{n-1}{r}\tilde{\phi}_r-|\tilde{\phi}_r|
=
\left\{
\begin{array}{ll}
-\left(\frac{n-1}{r}+1\right)\tilde{\phi}_r &\text{if} \ \tilde{\phi}_r\ge0, \\
-\left(\frac{n-1}{r}-1\right)\tilde{\phi}_r &\text{if} \ \tilde{\phi}_r<0,
\end{array}
\right. 
\]
we easily see that $-\frac{n-1}{r}\tilde{\phi}_r-|\tilde{\phi}_r|\le 0$ for $r \geq n-1$. 
In particular, $\tilde{\phi}$ is a supersolution to
\[
\tilde{\phi}_t \geq h(r) \quad \text{ in } (n-1,\infty) \times (0,\infty).
\]
We use the above and the fact that $h(r)=0$ for $r\in \tilde \cA$ to yield
\begin{equation}\label{monotone-phi}
t \mapsto \tilde{\phi}(r,t) \quad \text{ is nondecreasing for } r \in \tilde \cA.
\end{equation}

We now find upper and lower bounds for $\tilde \phi$.
Let $\psi$ be a solution to \eqref{eq:st-rad} constructed in Theorem \ref{thm:ex}, and $v(x)=\psi(|x|)$ for $x\in \R^n$.
Since $\psi$ is bounded from above, there exists $M>\|u_0\|_{L^\infty}$ such that $v-M\le u_0 \leq M$ on $\R^n$. 
By the usual comparison principle, we have 
\[
v(x)-M+ct \leq u(x,t) \leq M+ct  \quad \text{ on } \R^n \times[0,\infty).
\] 
Hence,
\[
\psi(r) - M \leq \tilde \phi(r,t) \leq M \quad \text{ on } [0,\infty) \times [0,\infty).
\]
Therefore, we are able to define the functions $\psi^{\pm}\in C([0,\infty))$ as
\begin{align*}
&\psi^{+}(r):=\limsup_{t\to\infty}{}^{\ast}\,\tilde{\phi}(r,t)
=\lim_{t\to\infty}\sup\left\{\tilde{\phi}(z,s)\,:\, |z-r|\le \frac{1}{s}, s\ge t\right\}, \\
&\psi^{-}(r):=\liminf_{t\to\infty}{}_{\ast}\,\tilde{\phi}(r,t) 
=\lim_{t\to\infty}\inf\left\{\tilde{\phi}(z,s)\,:\, |z-r|\le \frac{1}{s}, s\ge t\right\}, 
\end{align*}
which are, respectively, a subsolution and a supersolution to \eqref{eq:st-rad}. 
By \eqref{monotone-phi}, we have $\psi^{+}=\psi^{-}$ on $\tilde\cA$. 
In light of Theorem \ref{thm:unique}, we get the conclusion. 
\end{proof}

\begin{proof}[Proof of Theorem {\rm\ref{thm:main2}}]
In light of \eqref{monotone-phi}, for any $r\in\tilde{\cA}$, we have 
\begin{align*}
\phi(r,t)-ct\to&\,  
\sup_{\gam\in\cC(t,0;r,\rho),\rho>0,t>0}\left\{\int_0^{t} \left( \tilde{f}(\gam(\tau))-c \right)\,d\tau+\tilde{u}_0(\rho)\right\}\\
=&\, 
v_0(r) 
\qquad \text{ as } t\to\infty. 
\end{align*}
Due to the uniqueness set property of $\tilde\cA$ and the concavity of equation \eqref{eq:st-rad} with respect to $\psi_r$, the function $\psi_\infty$ is represented by \eqref{func:psi}. 
We refer to \cite[Theorem 3.1]{DS}, \cite[Theorem 1.3]{MT} (see also \cite[Theorem 5.24]{LMT}) for more details. 
\end{proof}

\subsection{A generalization}
In this subsection, we remove the radial symmetric assumption of initial data $u_0$ in a weak sense, and give a result on the convergence which slightly generalizes that of Theorem \ref{thm:main1}.

We consider the following initial data 
\[
w_0(x)=\tilde{u}_0(|x|)+\phi(x),  
\]
where $\tilde u_0\in\BUC([0,\infty))$, and $\phi\in\BUC(\R^n)$ are given. 
For $r \geq 0$, set 
\begin{align*}
\ol{\phi}(r):=\max_{|x|=r}\phi(x), \quad \text{and} \quad
\ul{\phi}(r):=\min_{|x|=r}\phi(x). 
\end{align*}
Define 
\begin{align}
&w_0^{+}(x):=\tilde{u}_0(|x|)+\ol{\phi}(|x|), \quad
w_0^{-}(x):=\tilde{u}_0(|x|)+\ul{\phi}(|x|), 
\label{func:w}\\ 
&
v_0^{\pm}(r):=\max\left\{d(r,t)+w_0^{\pm}(t)  \,:\,  t\in[0,\infty)\right\}, 
\label{func:v}
\end{align}
where $d$ is given by \eqref{func:d}. 
Here is a straightforward consequence of Theorem \ref{thm:main2}. 
\begin{prop}\label{prop:nonrad}
Assume that 
\begin{equation}\label{assump:nonrad}
v_0^{+}(s)=v_0^{-}(s)=:v_0(s) \quad \text{for all} \ s\in\tilde{\cA}.  
\end{equation}
Then,
\[
u(x,t)-c t\to \psi_\infty(|x|) \quad\text{locally uniformly for} \ x \in \R^n \ \text{as} \ t\to\infty, 
\]
where $c$ is given by \eqref{asym-speed}, and
\[
\psi_\infty(r)=\max\left\{d(r,s)+v_0(s)\,:\, s\in\tilde{\cA} \right\}.
\] 
\end{prop}

We give two nontrivial examples satisfying condition \eqref{assump:nonrad} in Proposition \ref{prop:nonrad}. 

\begin{ex}
{\rm
Let $\phi \in C_c(B(0,n-1))$, and $w_0^{\pm}$, $v_0^{\pm}$ be the functions defined by \eqref{func:w}, \eqref{func:v}, respectively. 
Then, we can easily see that \eqref{assump:nonrad} holds.  
Indeed, noting  
if $\gam(z) \in (0,n-1)$, then
\[
\gam'(z) \leq 1- \frac{n-1}{\gam(z)} <0, 
\]
which means that we cannot reach $s \in \tilde \cA$ if we start from $t\in (0,n-1)$, 
and thus $d(s,t)=-\infty$ for all $s\in \tilde \cA$, and $t \in (0,n-1)$. 
This example says that information of initial data in $B(0,n-1)$ does not matter 
for the large time behavior for (C) if $f$ and $\tilde u_0$ are radially symmetric. 
} 
\end{ex}

\begin{ex}
{\rm
Assume that $c:=\max_{r\ge1}\tilde{f}(r)>0$ and $\|u_0\|_{L^\infty} <R$ for some $R>0$, 
where $R$ is given by \eqref{f-source}. 
Let $\phi \in \BUC(\R^n)$ be a function satisfying
\[
\|\phi\|_{L^\infty} <R \quad \text{and}\quad \phi=0 \text{ in } B\left(0,R+\frac{6R}{c}\right).
\]
Let $w_0^{\pm}$, $v_0^{\pm}$ be the functions defined by \eqref{func:w}, \eqref{func:v}, respectively. 
Then, assumption \eqref{assump:nonrad} holds.  

Note first that we have $w_0^+(t)=w_0^-(t)$ for all $t\in[0,R+6R/c]$. 
Thus, we just need to prove that the maximum of the left hand side of \eqref{func:v} does not achieve for $t>R+6R/c$. 
Fix $t>R+6R/c$, and let $\gam$ be an admissible curve in \eqref{rep-phi} with 
$\gam(0)=t$, $\gam(T)=s$. 
If $\gam(r)>n-1$, then 
\[
-2 \leq -1-\frac{n-1}{\gam(r)} \leq \gam'(r) \leq 1- \frac{n-1}{\gam(r)}.
\]
Hence, it takes $\gam$ at least $3R/c$ time spending outside of $B(0,R)$ 
since $s\in\cA \subset [n-1,\infty)$.  
This gives that
\[
\int_0^T (\tilde{f}(\gam(r)) - c )\,dr \leq -\frac{3R}{c} c=-3R.
\]
Therefore, $d(s,t) \leq -3R$, and thus,
\[
d(s,t)+w_0^{\pm}(t) \leq -3R + R =-2R <w_0^{\pm}(s)=d(s,s)+w_0^{\pm}(s), 
\]
which implies \eqref{assump:nonrad}. 
}
\end{ex}

\section{Case 2: Non radially symmetric setting (a toy case)} \label{sec:case2}
In this section, we \textit{always} assume that (A3)--(A5) hold. 
Due to (A3), the asymptotic speed is given by $c=\max_{\R^n}f>0$ (see \cite{GMOT}).

\subsection{Ergodic problem \eqref{eq:static}}

\begin{lem}\label{lem:simple1}
Assume that {\rm (A3)} holds. 
Ergodic problem \eqref{eq:static} has a solution which is bounded from above.
\end{lem}

\begin{proof}
Set 
\[
\overline{f}(r):= \max_{|x|=r} f(x), \quad 
\underline{f}(r):= \min_{|x|=r} f(x) \quad \text{ for all } r\geq 0. 
\]
Without loss of generality, we can assume that $B(0,n-1)\subset U$. 
Then, there exist $n-1\leq a \leq b$ such that $\underline{f} \leq \overline{f}$ and
\[
\begin{cases}
\underline{f}(r) =c  \text{ for all } r \in [0,a], \text{ and } \underline{f}(r) <c \text{ for all } r>a,\\ \smallskip
\overline{f}(r) =c  \text{ for all } r \in [0,b], \text{ and }  \overline{f}(r) <c \text{ for all } r>b.
\end{cases}
\]
Thanks to the construction in Theorem \ref{thm:ex},  
there exist solutions $\overline{v}$ and $\underline{v}$, respectively, to \eqref{eq:static} with $f(x)=\overline{f}(|x|), \underline{f}(|x|)$ and $c=\max_{\R^n}f$, which are radially symmetric and bounded from above. 

By adding a constant, we have further that
\[
\underline{v} \leq \overline{v} \ \text{on} \ \R^n, \quad \text{and} \quad 
\underline{v}(x)=\overline{v}(x) = 0 \text{ in } B(0,a).
\]
We thus are able to use the Perron method to construct a solution $v$ to \eqref{eq:static} by 
\[
v(x) = \sup\left\{w(x)\,:\, \text{$w$ is a subsolution to \eqref{eq:static} and } \underline{v} \leq w \leq \overline{v} \ \text{on} \ \R^n \right\}
\]
for $x\in \R^n$. 
\end{proof}

Let $\overline{v}, \underline{v}$ be the functions defined in the proof of Lemma \ref{lem:simple1}. 
It is important noting that, as $f$ is compactly supported, 
$h(r)=-c$ for $r\gg 1$. Thus, integrating the equation \eqref{def:neg} with respect to $r$ yields  
\begin{equation}\label{behavior-v1-v2}
-c|x| -c(n-1) \log(|x|+1) - C \leq \underline{v}(x) \leq \overline{v}(x) \leq  -c|x|+C \quad \text{for}\ x\in \R^n.
\end{equation}
Define
\[
\Gam_f=\{v: \R^n \to \R \,:\, v \text{ satisfies } \underline{v} \leq v \leq \overline{v}\}.
\]
We have the following uniqueness result of solutions to \eqref{eq:static} in $\Gam_f$.

\begin{prop}\label{prop:unique}
Assume that $v,w \in \Gam_f$ are, respectively, a subsolution and a supersolution to \eqref{eq:static}.
Assume further that $v \leq w$ on $U$. Then $v \leq w$ in $\R^n$.
\end{prop}

Note that considering solutions in $\Gam_f$ is enough for our purpose of studying large time behavior to (C).

\begin{proof}
Assume by contradiction that there exists $y\in \R^n \setminus U$ such that $v(y)>w(y)$.
We can find $\lam>1$, which is close enough to $1$, such that
\[
0<\max_{\R^n} \left( \lam v - w \right) = \lam v(z) - w(z)
\]
for some $z\in \R^n \setminus U$. Note that $f(z)<c$.

For $\ep>0$, thanks to \eqref{behavior-v1-v2} and the fact that $\lam>1$, we are able to find $x_\ep,y_\ep \in \R^n$ such that
\[
\max_{x,y \in \R^n} \left( \lam v(x) - w(y) - \frac{|x-y|^4}{\ep} \right) = \lam v(x_\ep) - w(y_\ep) - \frac{|x_\ep-y_\ep|^4}{\ep},
\]
and, up to passing to a subsequence,
\[
\lim_{\ep \to 0} (x_\ep, y_\ep)=(z,z) \quad \text{and} \quad \lim_{\ep \to 0}  \frac{|x_\ep-y_\ep|^4}{\ep}=0.
\]
By the Crandall--Ishii Lemma \cite[Lemma 3.2]{CIL}, there exist $X_\ep, Y_\ep \in \mathbb{S}^n$ satisfying
\begin{equation}\label{CI}
 \begin{pmatrix} X_\ep &0\\ 0&-Y_\ep \end{pmatrix} 
\leq J +\ep J^2,
\end{equation}
where
\[
J=\frac{4}{\ep}\begin{pmatrix} Z &-Z\\ -Z&Z \end{pmatrix} \quad \text{and} \quad
Z=|x_\ep - y_\ep|^2 I_n + 2 (x_\ep-y_\ep)\otimes (x_\ep-y_\ep).
\]
We claim that for $\ep>0$ sufficiently small, $x_\ep \neq y_\ep$.
Indeed, assume otherwise that $x_\ep=y_\ep$, which is quite close to $z$. Then $f(x_\ep)<c$.
Since $\lam v(x) - \frac{|x-x_\ep|^4}{\ep}$ has a max at $x_\ep$, we employ the viscosity subsolution test to yield
\[
-\text{tr}((I_n-q\otimes q)0_{n})  \leq \lam(f(x_\ep)-c) <0.
\]
Here, $0_{n}$ is the zero matrix of size $n$, and $q \in \R^n$ is a vector such that $|q| \leq 1$.
We get a contradiction immediately.

Therefore,  for $\ep>0$ sufficiently small, $x_\ep \neq y_\ep$.
By viscosity subsolution and viscosity supersolution tests, we have
\begin{equation}\label{sub-v}
-\text{tr}((I_n-p\otimes p)X_\ep) - \frac{4|x_\ep - y_\ep|^3}{\ep} \leq \lam(f(x_\ep)-c),
\end{equation}
and
\begin{equation}\label{super-w}
-\text{tr}((I_n-p\otimes p)Y_\ep) - \frac{4|x_\ep - y_\ep|^3}{\ep} \geq f(y_\ep)-c.
\end{equation}
Here, $p = \frac{x_\ep - y_\ep}{|x_\ep -y_\ep|}$.
Combine \eqref{CI}--\eqref{super-w} to yield
\[
0\leq -\text{tr}((I_n-p\otimes p)(X_\ep-Y_\ep)) \leq \lam(f(x_\ep)-c) - (f(y_\ep)-c).
\]
Let $\ep \to 0$ to get 
\[
0 \leq (\lam-1) ( f(z)-c) <0,
\]
which is absurd.
\end{proof}

\subsection{Large time behavior}\label{subsec:large2}
In this subsection, under (A3)--(A5), we give the proof of Theorem \ref{thm:nonrad}.

\begin{lem}\label{lem:subsln}
Assume {\rm (A3)--(A5)}. 
Let $u$ be the viscosity solution to {\rm (C)}.
Then,
\[
u(x,t) = ct \quad \text{ for all } (x,t) \in U \times [0,\infty).
\]
\end{lem}

\begin{proof}
It is clear that $u(x,t) \leq ct$ for all $(x,t) \in \R^n \times [0,\infty)$.

To show the reverse inequality for $(x,t) \in U \times [0,\infty)$, we simply show that, for each $i\in I$, the following function
\[
\varphi(x,t) = ct \mathbf{1}_{E_i}(x) \quad \text{ for } (x,t) \in \R^n \times [0,\infty)
\]
is a subsolution to (C), where $\mathbf 1_E$ denotes the characteristic function of $E$, i.e., $\mathbf 1_E(x)=1$ for $x \in E$ and $\mathbf 1_E(x)=0$ 
if $x \in E^c$.
Indeed, take a smooth test function $\psi$ such that $\varphi-\psi$ has a strict global maximum at $(x_0,t_0) \in \R^n \times (0,\infty)$ and $\varphi(x_0,t_0)=\psi(x_0,t_0)$.
We only need to check the case that $x_0 \in \partial E_i$ as other cases are straightforward.

First of all, it is clear that $\psi_t(x_0,t_0) \leq c$. If we have that $D\psi(x_0,t_0)=0$, then $D^2\psi(x_0,t_0) \geq 0$.
Thus, for vector $\eta = 0 \in \R^n$, we have
\[
\psi_t(x_0,t_0) - (\del_{ij} - \eta_i \eta_j)\psi_{x_i x_j}(x_0,t_0) \leq c  - \Del \psi(x_0,t_0) \leq c = f(x_0).
\]
Secondly, if $D\psi(x_0,t_0) \neq 0$, then $D\psi(x_0,t_0)$ is a normal vector to the level set $\left\{x \in \R^n\,:\, \psi(x,t_0)=\psi(x_0,t_0)\right\}$.
Furthermore, as $\varphi-\psi$ has a strict global maximum at $(x_0,t_0)$, we deduce that
\[
E_i \subset W=\left\{ x\in \R^n\,:\, \psi(x,t_0) \geq \psi(x_0, t_0) \right\}.
\]
Therefore, the mean curvature to $W$ at $x_0$ is smaller than or equal to the mean curvature to $E_i$ at $x_0$.
In light of (A4),
\[
-\text{div}\left(\frac{D\psi(x_0,t_0)}{|D\psi(x_0,t_0)|} \right) \leq 1,
\]
which gives
\[
\psi_t(x_0,t_0) - \left(\text{div}\left(\frac{D\psi(x_0,t_0)}{|D\psi(x_0,t_0)|} \right) + 1 \right) |D\psi(x_0,t_0)| \leq \psi_t(x_0,t_0)  \leq c=f(x_0). 
\qedhere
\]
\end{proof}

\begin{rem} \ \\
(i) By \cite[Lemma 5.1]{GMT}, we see that the set $\Gam(t)=\{x\in\R^n\,:\, u(x,t)=ct\}$ moves according to the obstacle problem of the surface evolution $V=\kap+1$ with obstacle $U$. 
Lemma \ref{lem:subsln} is consistent with this observation. \\
(ii) Let us consider a front propagation problem $V=1$ with a source term $f$. 
The corresponding equation is
\[
u_t-|Du|=f(x) \quad\text{in} \ \R^n \times (0,\infty),  
\]  
which is a typical example in \cite{NR, BLNP}. 
Then, the set $U=\{x\in \R^n\,:\,f(x)=\max_{\R^n} f\}$  plays a role of not only  a \textit{uniqueness set} for the associated ergodic problem 
but also a \textit{monotonicity set} of $u$, that is, 
$t\mapsto u(x,t)-ct$ is monotone for every $x\in U$. 

It is important emphasizing here that, in our problem, $U$ in general is not a monotonicity set  of $u$ if we assume only (A3).
Additional assumptions (A4)--(A5) are really needed in this current approach.
\end{rem}

\begin{proof}[Proof of Theorem {\rm{\ref{thm:nonrad}}}]
For $x\in \R^n$, set
\[
v(x)=\limsup_{t\to \infty}{}^* \left(u(x,t)-ct \right),\quad
w(x)=\liminf_{t\to \infty}{}_* \left(u(x,t)-ct \right).
\]
It is clear then that $v,w \in \Gam_f$ and $v,w$ are a subsolution and a supersolution to \eqref{eq:static}, respectively.
By Lemma \ref{lem:subsln}, we get further that $v=w=0$ on $U$.
We then use Proposition \ref{prop:unique} to yield the desired result.
\end{proof}

\begin{rem}
In fact, in this section, we can relax {\rm (A4)} by the following weaker assumption.
\begin{itemize}
\item[(A4')] $U=\cup_{i\in I} E_i$, where $I \neq \emptyset$ is an index set, and for each $i\in I$, $E_i \subset \R^n$ is a closed set with nonempty interior  satisfying $-\kappa-1 \leq 0$ on $\partial E_i$ in the sense of viscosity solutions to be defined in the following discussion.
\end{itemize}
This relaxation is pretty clear to see as one only needs to use (A4') in the proof of Lemma \ref{lem:subsln}.
\end{rem}

\noindent Let us now explain clearly assumption (A4').
We first recall the notion of set-theoretic solution $\{D_t\}_{t\ge0}\subset\R^n$ to the surface evolution equation 
\begin{equation}\label{eq:surface}
V=\kappa+1 \qquad \text{ on }   \partial D_t,  
\end{equation}
where $V$ and $\kappa$ denote the outward normal velocity and the outward mean curvature of $\partial D_t$, respectively. 

\begin{defn}\label{defn:set-sol}
We say that an evolving family $\{D_t\}_{t\ge 0}$ of open sets $D_t$ in $\mathbb{R}^n$ is a  set-theoretic subsolution (resp., supersolution) of \eqref{eq:surface} if $\mathbf{1}_{D_t}$ is a viscosity subsolution (resp., supersolution) of 
\begin{equation} \label{LE}
	u_t - \left(\operatorname{div}\left( \frac{Du}{|Du|} \right)+1 \right) |Du| = 0
	\quad\text{in}\quad	\R^n \times (0,\infty), 
\end{equation}
where $\mathbf{1}_{D_t}$ is a characteristic function of $D_t$, i.e., 
$\mathbf{1}_{D_t}(x)=1$ if $x\in D_t$ and 
$\mathbf{1}_{D_t}(x)=0$ if $x\not\in D_t$. 

We say that $\{D_t\}_{t\ge 0}$ is a  set-theoretic solution of \eqref{eq:surface} if it is both a set-theoretic subsolution and a set-theoretic supersolution. 
\end{defn}

Next, we give the definition of stationary set-theoretic solution for a closed set.
Of course, the definition of stationary set-theoretic solution for an open set follows in the same way.

\begin{defn}\label{defn:set-sta}
Let $E \subset \R^n$ be  a closed and bounded set with nonempty interior such that $\partial E = \partial (E^\circ)$.
Here, $E^\circ$ is the interior of $E$.

We say that $E$ satisfies $-\kap-1\le0$ on $\partial E$ (resp., $-\kap-1\ge 0$ on $\partial E$) in the sense of viscosity solutions if $\{D_t\}_{t \geq 0}$ with $D_t=E^\circ$ for all $t \geq 0$ is a set theoretic subsolution (resp., supersolution) to \eqref{LE}.

We say that $E$ satisfies $-\kap-1= 0$ on $\partial E$ if both $-\kap-1\le0$  and $-\kap-1\ge0$ on $\partial E$ hold.

\end{defn}

\section{Some further discussions}\label{sec:further}
In this section, we consider a particular case of $U$ in two dimensions, which is given by 
a stadium shape, that is,
\begin{equation}\label{U-2}
U = \bigcup_{x \in [-a,a]} \ol{B}((x,0),1),
\end{equation}
for some fixed $a>0$.

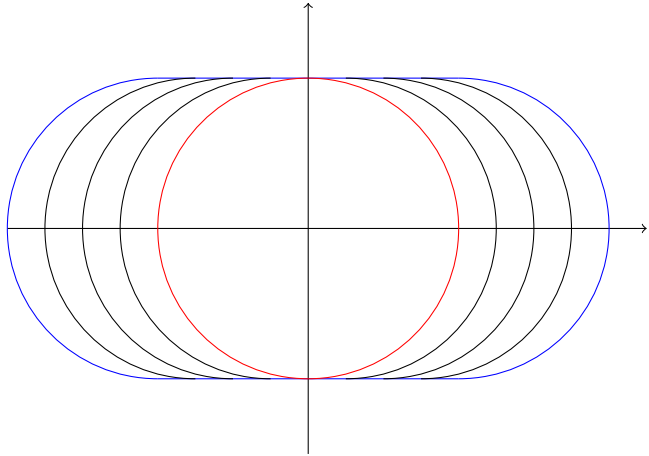
\begin{figure}[h]
\begin{center}
\begin{tikzpicture}[baseline=(current bounding box.north)]

\draw[blue] (-2,-2)--(2,-2);
\draw[blue] (-2,2)--(2,2);

\draw[->] (-4,0)--(4.5,0);
\draw[->] (0,-3)--(0,3);

\draw[red] (0,-2) arc(-90:90:2);
\draw (0.5,-2) arc(-90:90:2);
\draw (1,-2) arc(-90:90:2);
\draw (1.5,-2) arc(-90:90:2);

\draw[red] (0,-2) arc(90:-90:-2);
\draw (-0.5,-2) arc(90:-90:-2);
\draw (-1,-2) arc(90:-90:-2);
\draw (-1.5,-2) arc(90:-90:-2);

\draw[blue] (2,-2) arc(-90:90:2);
\draw[blue] (-2,-2) arc(90:-90:-2);

\end{tikzpicture}
\caption{$U$ and semicircles of radii $1$} \label{fig:U}
\end{center}
\end{figure}

Let us now consider \eqref{eq:static} in $U$ only. 
Since $U$ satisfies (A3), we have $c=\max_{\R^2}f$. 
Thus, the stationary equation \eqref{eq:static} in the interior of $U$ reads
\begin{equation}\label{eq:v-U}
-\left(\Div\Big(\frac{Dv}{|Dv|}\Big)+1\right)|Dv|=0 \quad \text{ in }  U^\circ.  
\end{equation}
Here, $U^\circ$ is the interior of $U$.
%
%

We first construct particular solutions to \eqref{eq:v-U}. 
\begin{lem} \label{lem:non-unique}
Assume $n=2$ and \eqref{U-2}.
Then, there exist infinitely many solutions $v\in C(U)$ of \eqref{eq:v-U} such that $v=0$ in $B(0,1)$ and $v  \neq 0$.
\end{lem}

\begin{proof}
We provide an explicit construction of a solution $v$ as following.

For $x\in U$ such that $x_1 \leq 0$ or $|x| \leq 1$, we set $v(x)=0$.
For $x\in U$ such that $x_1>0$ and $|x|>1$, we set
\[
v(x) = \sqrt{1 - x_2^2} - x_1.
\]
Let us explain a bit the definition of $v$ in the later part. 
For any $0\le c\le a$, setting $\Gam_c:=\{(x_1,x_2)\,:\, \sqrt{1-x_2^2}-x_1=-c\}$, 
we can easily see that $\Gam_c\subset\partial B(c,1)$. 
Thus, geometrically, if $x$ stays on a semicircle of radius $1$ with center at $(0,c)$ for $c>0$, then $v(x)=-c$.

One can check analytically that $v$ is a solution to \eqref{eq:v-U} in a straightforward way.
Let us instead give a geometrical explanation here.
Indeed, for each $b<0$, the level set $\{x\,:\,v(x)=b\}$ is exactly a semicircle of radius $1$.
Thus, for $x \in \{x\,:\, v(x)=b\}$, the mean curvature of the level set at $x$ is exactly $1$, which means
\[
-\Div\left(\frac{Dv(x)}{|Dv(x)|}\right) = 1,
\]
which confirms that $v$ is a solution to \eqref{eq:v-U}.

Now, to create infinitely many such solutions, we just need to define
\[
\tilde v(x) = \phi(v(x)) \quad \text{ for all } x \in U,
\]
where $\phi:\R \to \R$ is any smooth function such that $\phi(0)=0$.
Then $\tilde v$ is always a solution to \eqref{eq:v-U}.
The proof is complete.
\end{proof}

Next, we give a result to characterize a solution obtained in Lemma \ref{lem:non-unique}.
\begin{lem} \label{Ch}
Let $u$ be a solution of \eqref{eq:v-U} which is continuous up to the boundary of $U$.
Assume that $u(x)=0$ for $x\in U$ such that $x_1 \leq 0$ or $|x| \leq 1$.
Assume further that $u$ is strictly decreasing in $x_1$ for $x_1>0$ at the boundary, 
that is, $u(x_1, x_2)>u(\tilde x_1, x_2)$ for any $(x_1, x_2), (\tilde x_1, x_2)\in\partial U$ with $0<x_1<\tilde x_1$. 
Then, the level curves of $u$ for $x_1>0$ agree with those of $v(x)=\sqrt{1-x^2_2}-x_1$.
\end{lem}

To prove Lemma \ref{Ch}, we consider a set-theoretic solution of $-\kappa-1=0$ in $U$ defined in Definition \ref{defn:set-sta}.   
\begin{thm} \label{Co}
Assume that $U$ is open and convex in $\mathbb{R}^2$.
Let $D\subset U$ be an open set and a  set-theoretic solution of $-\kappa-1=0$ on $\partial D \cap U$.
Let $D_i$ be a connected component of $D$.
Then $D_i$ is convex.
Let $S_i$ be a connected component of $\partial D_i$ in $U$. 
Then $S_i$ is an open arc of a unit circle.
\end{thm}
For the proof, we first derive convexity of $D$ if $D$ is connected.
\begin{lem} \label{CV}
Assume that $U$ is open and convex in $\mathbb{R}^2$.
Let $D\subset U$ be an open set and a   set-theoretic solution of $-\kappa-1\geq 0$  on $\partial D \cap U$.
 Then each connected component of $D$ must be convex in $U$. 
\end{lem}
\begin{proof}
Assume that $D$ is not convex but connected.
 Then there are points $P_1, P_2 \in D$ such that the line segment $P_1 P_2$ intersects $U \setminus \overline{D}$.
 We may assume that $P_1 P_2$ is a closed, nontrivial interval of the $x_1$-axis.
 Without loss of generality, we may assume further that there is  a connected component $K$ of the closed set $(\ol U \setminus D) \cap \{x_2 \leq 0\}$ which contains some part of the segment $P_1 P_2$.
 The set $K$ is compact since $U$ is convex and $D$ is connected. 

Clearly, there is a point $Q=(x_1^*,x_2^*) \in \partial D$ such that
\[
	x_2^* = \inf \left\{ x_2 \in \mathbb{R} \,:\,(x_1,x_2) \in K \text{ for some } x_1\in\mathbb{R} \right\}.
\]
At the point $Q$, one is able to use the straight line $x_2=x_2^*$ as a test function to yield a contradiction since $\mathbf{1}_D$ must be a supersolution of \eqref{LE}.
 The convexity of each connected component of $D$ follows right away and the proof is now complete.
\end{proof}
\begin{lem} \label{Gr}
Assume that $U$ is open and convex in $\mathbb{R}^2$.
Let $D \subset U$ be an open, convex set and  a set-theoretic solution of $-\kappa-1=0$ on $\partial D \cap U$.
 Then, the boundary of $D$ in $U$ consists of open arcs of unit circle. 
\end{lem}
\begin{proof}
Since $D$ is convex, $\partial D$ is locally represented as the graphs of concave functions.
 If $\partial D$ is locally represented as the graph of a concave function $x_2=f(x_1)$, then $f$ must satisfy the graph equation of $-\kappa-1=0$ in viscosity sense and $D$ is locally represented as $\{(x_1,x_2)\,:\, x_2>f(x_1)\}$ by applying results of \cite[Section 5]{GG}.
 Since $f$ is concave, it is twice differentiable a.e. thanks to Alexandrov's theorem for convex functions.
 Then, $f$ satisfies
\[
	-\frac{d}{dx} \frac{f'(x)}{\left(1+\left(f'(x)\right)^2\right)^{1/2}} -1 = 0
\]
almost everywhere.
 This implies that $x_2=f(x_1)$ must be a part of the graph of $\sqrt{1-x^2_1}$ up to translation.
\end{proof}

Theorem \ref{Co} follows immediately from Lemma \ref{CV} and Lemma \ref{Gr}.
 If $U=\mathbb{R}^2$, then by Theorem \ref{Co}, $D_i$ is convex so that $S_i=\partial D_i$ and $S_i$ is an open arc of a unit circle.
 This means that $D_i$ is an open unit disk.
It is worth emphasizing that this answers a question raised by \cite[Problem C]{GTZ} at least partially. 
\begin{cor} \label{AL}
Let $D \subset \R^2$ be an open set and a set-theoretic solution of $-\kappa-1=0$ on $\partial D$.
 Then $D$ is an open unit disk if $D$ is connected.
 In general, each connected component of $U$ is an open unit disk. 
\end{cor}

As an application of Theorem \ref{Co}, we consider a kind of Alexandrov's problem in $\mathbb{R}^2$ in viscosity sense.
 Assume that $\{D_t\}_{t\in I}$ is an open evolution of $V=\kappa+1$ and $E_t=\overline{D_t}$ is a closed evolution for $V=\kappa+1$ as in \cite{G-book}.
 In other words, there is a uniformly continuous function $u$ in $\mathbb{R}^2 \times \overline{I}$ such that $u$ solves \eqref{LE} in $\mathbb{R}^2 \times I$ in viscosity sense with the property that
\begin{align*}
	E_t &= \left\{ x \subset \mathbb{R}^2 \,:\, u(x,t) \geq 0 \right\} \\
	D_t &= \left\{ x \in \mathbb{R}^2 \,:\, u(x,t) > 0 \right\}.
\end{align*}
for $t\in I$.
 In general, the set $\{u\geq 0\}$ can be strictly larger that $\overline{\{u>0\}}$ even if initially $\left\{ x \,:\, u(x,0)\geq 0 \right\}$ is the closure of $\left\{ x \,:\, u(x,0)>0 \right\}$.
 Such a phenomenon is called fattening as pointed out by \cite{ES1} for the mean curvature flow (see also \cite{G-book}).
\begin{thm} \label{Alex}
Let $D$ be a stationary set-theoretic solution of $V=\kappa+1$ in $\mathbb{R}^2$.
 Assume that $\overline{D}$ is a closed evolution for $V=\kappa+1$ in $\mathbb{R}^2$.
 Then $D$ consists of a family of unit disks whose closures are mutually disjoint. 
\end{thm}
\begin{proof}
Because of Corollary \ref{AL} each component of $D$ is an open unit disk.
 If the closure touches, it is known that a fattening phenomenon occurs as pointed out by \cite{Z}.
 Then $E=\overline{D}$ cannot be a closed evolution.
 This is a contradiction, and hence, the closure of each connected component does not intersect.
 The proof is now complete. 
\end{proof}
\begin{proof}[Proof of Lemma {\rm\ref{Ch}}]
By Theorem \ref{Co}, we see that a negative level set of $u$ in $U$ consists of  open arcs of unit circles.
 The endpoints of these arcs must be on the boundary of $U$  ($(x_1,x_2)$ with $0<x_1\leq a$ and $x_2=\pm 1$, and the right most semicircle of $\partial U$).
 Both endpoints of an arc cannot be  on $x_2=1$ because of the assumption on the monotonicity of $u$ on the boundary.
 Similarly, both endpoints of an arc  cannot be on $x_2=-1$.
Also, it is not possible that, for a given arc, one endpoint is on $x_2=1$ and one endpoint is on the right most semicircle of $\partial U$ as then $u$ cannot be continuous because each level set nearby must be an arc of a unit circle that has the same property (one endpoint is on $x_2=1$ and one endpoint is on the right most semicircle of $\partial U$), and two of these level sets must intersect in $U$.

 Thus, we conclude that each negative level curve of $u$ contains a semicircle with one endpoint on $x_2=1$ and the other endpoint on $x_2=-1$ or an arc of a unit circle with two endpoints both on the right most semicircle of $\partial U$.
 Since $u(x)=0$ for $x\in U$ such that $x_1 \leq 0$ or $|x| \leq 1$, and $u$ is strictly decreasing in $x_1$ for $x_1>0$ at the boundary, 
 the second scenario cannot happen.
Therefore, the level curves  of  $\{u(x)=c\}$ for $c<0$ give the same foliation as those of $v$.

 Note that one can only compare $u$ with $v$ by its level curves because all functions $\theta\circ v$ with nondecreasing continuous $\theta$ such that $\theta(0)=0$ solve \eqref{eq:v-U}.
\end{proof}

We now conclude the paper by pointing out that 
the result of Lemma \ref{lem:non-unique} gives us some further intuition that it is quite complicated to get convergence result \eqref{conv} in the most general setting.
It gives us a fact that it is not easy to understand solutions of \eqref{eq:v-U}, and more generally, solutions of \eqref{eq:static}, and we probably need to be able characterize all solutions to \eqref{eq:v-U} in order to proceed further.


\bibliographystyle{amsplain}
\providecommand{\bysame}{\leavevmode\hbox to3em{\hrulefill}\thinspace}
\providecommand{\MR}{\relax\ifhmode\unskip\space\fi MR }
\providecommand{\MRhref}[2]{%
  \href{http://www.ams.org/mathscinet-getitem?mr=#1}{#2}
}
\providecommand{\href}[2]{#2}

\end{document}